\documentclass[12pt,reqno]{amsart}
\usepackage[utf8]{inputenc}
\usepackage[T1]{fontenc}
\usepackage{a4wide}
\usepackage{amsmath}
\usepackage{amsthm}
\usepackage{amsfonts}
\usepackage{amssymb}
\usepackage{graphicx}
\usepackage[stable]{footmisc}
\usepackage{xcolor}
\usepackage[colorlinks,linkcolor=teal,citecolor=teal]{hyperref}

\numberwithin{equation}{section}
\newtheorem*{remark}{Remark}
\newtheorem{definition}{Definition}
\newtheorem{lemma}{\bf Lemma}[section]
\newtheorem{prop}{\bf Proposition}[section]
\newtheorem{thm}{\bf Theorem}[section]

\begin{document}
	\title{Formally Integrable Structures II. Division Problem}
	\author{Qingchun Ji}
	\address{School of Mathematical Sciences, Fudan University, Shanghai 200433, China}
	\email{qingchunji@fudan.edu.cn}
	\author{Jun Yao}
	\address{School of Mathematical Sciences, Fudan University, Shanghai 200433, China}
	\email{jyao\_fd@fudan.edu.cn}
	\thanks{This work was partially supported by National Natural Science Foundation of China, No. 12071081 and No. 11801366.}
	\date{}

\begin{abstract}
	We formulate a division problem for a class of overdetermined systems introduced by L. H{\"o}rmander, and establish an effective divisibility criterion. In addition, we prove a coherence theorem which extends Nadel's coherence theorem from complex structures to elliptic systems of partial differential equations.
\end{abstract}
\maketitle

\section{Introduction}
As a generalization of the Newlander-Nirenberg theorem, L. Nirenberg introduced the notion of Levi flat structures and proved that every Levi flat structure is in fact locally integrable(\cite{N}). Later, L. H{\"o}rmander(\cite{H1}) considered, under involutivity conditions inspired by Levi flat structures, the following general overdetermined system for the unknown $u$
\begin{equation}\label{OS}
	P_\alpha u=f_\alpha, 1\leq \alpha\leq r,
\end{equation} 
where $P_\alpha$'s are first order differential operators in an open subset $ \Omega\subseteq\mathbb{R}^n $ given by
\begin{equation}\label{operator}
	P_\alpha=\sum_{\sigma=1}^{n}a_\alpha^\sigma\partial_\sigma+a_\alpha^0,
\end{equation}
where $a_\alpha^\sigma\in C^{\infty}(\Omega)$ for $0\leq \sigma\leq n,1\leq \alpha \leq r$ and $ \partial_\sigma = \partial/\partial x_\sigma $ for $ 1\leq \sigma\leq n $. In \cite{JYY}, we constructed a resolution of the solution sheaf $\mathcal{S}_P$ which consists of germs locally square-integrable solutions of the homogeneous system $Pu=0$ where $P:=(P_1,\cdots,P_r)$. We also formulated Cousin type problems for the system ({\ref{OS}) based on which we gave sheaf-theoretical approaches to gluing local solutions of ({\ref{OS}). For recent progress on Levi flat structures, we refer to \cite{G}, \cite{HT}, \cite{S}, \cite{Web},\cite{W} and references therein.



Skoda's result on ideal generation(see \cite{Djp82}, \cite{Sh72}, \cite{Sh78} and \cite{Syt18}) has played a great role in algebraic geometry and complex geometry. Siu(see \cite{Syt98,Syt02,Syt04,Syt05,Syt09}) used Skoda's theorem to establish the remarkable deformation invariance of plurigenera and proved the finite generation of canonical ring of compact complex algebraic manifolds of general type. There have been many generalizations of Skoda's division theorem, e.g., division theorems for the Koszul complex and the exact sequences of holomorphic vector bundles in \cite{Djp82}, \cite{Jqc12}, \cite{Jqc13} and \cite{Kd}, twisted version Skoda's estimate in \cite{Vd08}.

This paper is devoted to the division problem for a Koszul type complex of the solution sheaf $\mathcal{S}_P$ of (\ref{OS}) under the assumption that the operator $P$ is elliptic(Theorem \ref{division thm})
\begin{align*}
	0\longrightarrow\bigwedge^m\mathcal{S}_P^{\oplus m}\stackrel{\iota_h^\star}{\longrightarrow}\bigwedge^{m-1}\mathcal{S}_P^{\oplus m}\stackrel{\iota_h^\star}{\longrightarrow}\cdots\stackrel{\iota_h^\star}{\longrightarrow}\mathcal{S}_P^{\oplus m}\stackrel{\iota_h^\star}{\longrightarrow}\mathcal{S}_P\longrightarrow0,
\end{align*}
where $\iota_h^\star$ means the interior product with $h\in\mathcal{S}_P^{\oplus m}$(see (\ref{eoc.})). In addition, we show the coherence of $\left(\mathcal{S}_P,\star,+\right)$-subsheaves $\mathcal{I}_\phi^\cdot$ of $\bigwedge^\cdot\mathcal{S}_P^{\oplus m}$(Definition \ref{ideal}), which is a generalization of Nadel's coherence theorem for multiplier ideal sheaves with respect to complex structures(Proposition \ref{coherent}).

We will adopt the summation convention over repeated indices,  and arrange that the indices $\sigma,\delta$ run from $1$ to $n$, indices $i,j,k$ run from $1$ to $m$, indices $\alpha,\beta,\gamma$ run from $1$ to $r$, indices $\mu,\nu$ run from $1$ to $r-s$($s\leq r$ is an integer to be described later). Moreover, every multi-index $ J=(j_1,\cdots,j_q) $ has its  components in $\{1,\cdots,m\}$.

\section{Preliminaries}
Following \cite{H1}, we will always assume the following conditions on these differential operators (\ref{operator}) occurring in the overdetermined system (\ref{OS}) under consideration.
For all $ 1\leq \alpha,\beta\leq r$, \begin{align}
	\left[P_\alpha,P_\beta\right] &= c_{\alpha\beta}^\gamma P_\gamma \ \ {\rm and } \ 
	c_{\alpha\beta}^\gamma=-c_{\beta\alpha}^\gamma,\label{integrability condition}\\ 
	\left[p_\alpha,\bar{p}_\beta\right] & =d_{\alpha\beta}^\gamma p_\gamma-e_{\alpha\beta}^\gamma\bar{p}_\gamma, \label{condition of the Lie braket}
\end{align}
where  $ c_{\alpha\beta}^\gamma, d_{\alpha\beta}^\gamma, e_{\alpha\beta}^\gamma  \in C^{\infty}(\Omega) $,  $p_\alpha$ is the principal part of $P_\alpha$ and $ \bar{p}_\alpha $ is the complex conjugate of $ p_\alpha $, i.e.,
$\bar{p}_\alpha=\bar{a}_\alpha^\sigma\partial_\sigma(1\leq \alpha,\beta,\gamma\leq r)$. We also assume that $P=(P_1,\cdots,P_r)$ is elliptic in this paper.

Recall that for a real valued function $ \varphi\in C^2(\Omega) $  and $\mathfrak{p}\in \Omega$, the quadratic form $Q_{\varphi,\mathfrak{p}}$ on $ \mathbb{C}^{r}$ is defined as
	\begin{equation}
        Q_{\varphi,\mathfrak{p}}(\xi,\xi)={\rm Re}\left(p_\alpha\bar{p}_\beta\varphi(\mathfrak{p})+e_{\alpha\beta}^\gamma(\mathfrak{p})\bar{p}_\gamma\varphi(\mathfrak{p})\right)\xi_{\beta}\bar{\xi}_{\alpha}	\label{quadratic form}
    \end{equation}
    for all $ \xi=(\xi_\alpha)_{1\leq \alpha\leq r}\in\mathbb{C}^{r} $. If there exists an exhaustion function $\varphi\in C^2(\Omega)$ such that $Q_{\varphi,\mathfrak{p}}$ is positive define for every $\mathfrak{p}\in \Omega$, then $\Omega$ is said to be \emph{$P$-convex}. Given a point $\mathfrak{p}_0\in \Omega$, let $\varphi\in C^2(\Omega)$ be an arbitrary function vanishing at $\mathfrak{p}_0$ up to first order derivatives, it's observed in \cite{JYY} that
\begin{equation}
	Q_{\varphi,\mathfrak{p}_0}(\xi,\xi)=a_\alpha^\sigma(\mathfrak{p}_0)\xi_\alpha\overline{a_\beta^\delta(\mathfrak{p}_0)\xi_\beta}\partial_\sigma\partial_\delta \varphi(\mathfrak{p}_0)\label{local}.
\end{equation}
This observation will be used in the sequel. The ellipticity of the operator $P$ implies that
\begin{equation*}
	\mathcal{S}_P\subseteq\mathcal{E}
\end{equation*} 
where $\mathcal{E}$ is the sheaf of germs of smooth functions over $\Omega$. As noted in \cite{JYY}, when $\Omega$ is $P$-convex there is a function $\eta\in C^\infty(\Omega)$ such that for $ 1\leq \alpha\leq r$ 
\begin{equation*}
	p_\alpha\eta=a^0_\alpha,
\end{equation*}
i.e.,
\begin{equation}\label{eta}
	P_\alpha=e^{-\eta}p_\alpha(e^\eta \ \cdot).
\end{equation}
Fix such a function $\eta$, we can define a multiplication on $\mathcal{S}_P$ as we did in \cite{JYY}
\begin{equation}
	u\star v:=e^\eta uv,\label{star}
\end{equation}
which satisfies 
\begin{equation*} 
	P_\alpha(u\star v)=P_\alpha u\star v+u\star P_\alpha v, \ 1\leq j\leq r.
\end{equation*}
Together with the usual addition $+$, $\left(\mathcal{S}_P,\star,+\right)$  is a sheaf of rings.

\begin{definition}
	Let $\phi$ be a real-valued measurable function in $\Omega$ which is locally bounded from above. $\phi$ is said to be a $P$-convex function, if for any open subset $U\Subset\Omega$, there exists a decreasing sequence of real-valued functions $\{\phi_\iota\}_{\iota=1}^\infty$ in $C^2\left(U\right)$, such that the quadratic forms $\{Q_{\phi_\iota,\mathfrak{p}}\}_{\iota=1}^\infty$ defined by (\ref{quadratic form}) are all non-negative at every point $\mathfrak{p}\in U$, and $\phi_\iota\longrightarrow\phi$ almost everywhere in $U$ as $\iota\rightarrow\infty$.
\end{definition}

If $\phi$ is locally integrable in $\Omega$, then a $P$-convex function $\phi$ has nonnegative quadratic form $Q_{\phi,\mathfrak{p}}$ in the sense of distribution theory.
Given a $P$-convex function $\phi$ and $0\leq \ell\leq m$, we introduce the following subsheaf $\mathcal{I}_\phi^\ell$  of $\bigwedge^\ell\mathcal{S}_P^{\oplus m}$.

\begin{definition}\label{ideal}
	Let $\phi$ be a $P$-convex function in $\Omega$ and $0\leq \ell\leq m$,  we define $\mathcal{I}_\phi^\ell$  as the subsheaf  of $\bigwedge^\ell\mathcal{S}_P^{\oplus m}$  consisting  of germs $f\in \bigwedge^\ell\mathcal{S}_{P,\mathfrak{p}}^{\oplus m}$ such that $|f|^2e^{-\phi}$ is integrable in a neighborhood of $\mathfrak{p}\in\Omega$.
\end{definition}

\section{The division problem}
We begin by formulating the division problem. For an arbitrary $h=\left(h_1,\cdots,h_m\right)\in\Gamma\left(\Omega,\mathcal{S}_P\right)^{\oplus m}$, we have the following Koszul type complex
\begin{align}\label{qkc}
	0\longrightarrow\bigwedge^m\mathcal{S}_P^{\oplus m}\stackrel{\iota_h^\star}{\longrightarrow}\cdots
	\stackrel{\iota_h^\star}{\longrightarrow}\bigwedge^\ell\mathcal{S}_P^{\oplus m}\stackrel{\iota_h^\star}{\longrightarrow}\bigwedge^{\ell-1}\mathcal{S}_P^{\oplus m}\stackrel{\iota_h^\star}{\longrightarrow}\cdots\stackrel{\iota_h^\star}{\longrightarrow}\mathcal{S}_P\longrightarrow0.
\end{align}
For $1\leq \ell\leq m$, $u=\left(u_J\right)_{|J|=\ell}\in\bigwedge^\ell\mathcal{S}_P^{\oplus m}$ where $J$ are strictly increasing multi-indices, $\iota_h^\star:\bigwedge^\ell\mathcal{S}_P^{\oplus m}\rightarrow\bigwedge^{\ell-1}\mathcal{S}_P^{\oplus m}$ is defined as
\begin{align}\label{eoc.}
	\iota_h^\star u=\bigg(\sum_{j\notin I}h_j\star u_{jI}\bigg)_{|I|=\ell-1},
\end{align}
where $\star$ is defined by (\ref{star}). It's obvious that $\left(\iota_h^\star\right)^2=0$. In the case $\eta=0$, we denote $\iota_h^\star$ simply by $\iota_h$.

The global division problem for the complex (\ref{qkc}) means to find the sufficient condition of the global exactness of (\ref{qkc}), i.e., for a given $f\in\bigwedge^{\ell-1}\Gamma\left(\Omega,\mathcal{S}_P\right)^{\oplus m}$ with $\iota_h^\star f=0$, we are looking for a sufficient condition on $f$ for the existence of some $u\in\bigwedge^{\ell}\Gamma\left(\Omega,\mathcal{S}_P\right)^{\oplus m}$ such that $\iota_h^\star u=f$ which is equivalent to
\begin{align}\label{ee.}
	\iota_{e^\eta h} \left(e^\eta u\right)=e^\eta f.
\end{align}
The global division problem for $\iota_h^\star:\mathcal{S}_P^{\oplus m}\rightarrow\mathcal{S}_P$(i.e., $\ell=1$) corresponds to the ideal generation problem in the ring $\Gamma\left(\Omega,\mathcal{S}_P\right)$. 

On the other hand, by (\ref{eta}) we have
\begin{align}\label{Pp}
	\mathcal{S}_p=\{e^\eta g\ |\ g\in\mathcal{S}_P\},
\end{align}
where $p:=\left(p_1,\cdots,p_r\right)$  and $\mathcal{S}_p$ is the sheaf of germs of solutions with respect to the operator $p$. By definition, $e^\eta h\in\Gamma\left(\Omega,\mathcal{S}_p\right)^{\oplus m}$ and $e^\eta f\in\bigwedge^{\ell}\Gamma\left(\Omega,\mathcal{S}_p\right)^{\oplus m}$. Since $\mathcal{S}_p$ is a sheaf of rings with respect to the usual addition and multiplication, we can introduce the division problem for $p$ as above. According to (\ref{ee.}), if the division problem is solvable for $p$, then the original problem for $P$ can also be
solved. Hence, we restrict to the special case $\eta=0$(i.e., $p=P$).

In order to establish a division theorem for (\ref{qkc}), we will make use of the following Riesz representation theorem(see \cite{Sh72} and \cite{Vd08}).
\begin{lemma}\label{riesz}
	Let $H,H_1,H_2,H_3$ be Hilbert spaces, $R:H_1\longrightarrow H$ be a bounded linear operator, $ T:H_1\longrightarrow H_2 $ and $ S:H_2\longrightarrow H_3 $ be linear, closed, densely defined operators such that $ {\rm Im}T\subseteq {\rm Ker}S $. Let $D\subseteq H$ be a closed space satisfying $R\left({\rm Ker}T\right)\subseteq D$. Then for any $f\in D$ and constant $C>0$, the following statements are equivalent.
	\begin{enumerate}
		\item[(1)] There exists at least one $u\in{\rm Ker}T$ such that $Ru=f$ and $\|u\|_{H_1}\leq C$.
		
		\item[(2)] $|\left(f,v\right)_H|\leq C\|R^*v+T^*w\|_{H_1}$ hold for all $v\in D,\ \text{and} \ w\in{\rm Ker}S\cap{\rm Dom}\left(T^*\right)$.
	\end{enumerate}
\end{lemma}

For $1\leq \ell\leq m$, we will apply the lemma to the following Hilbert spaces and operators(note that we temporarily assume $\eta=0$): 
\begin{align*}
	H_1=\bigwedge^\ell L_{\phi_1}^2(\Omega)^{\oplus m},&\ H_2=\bigwedge^\ell L_{\phi_2}^2(\Omega)^{\oplus m\tbinom{r}{1} },\ H_3=\bigwedge^\ell L_{\phi_3}^2(\Omega)^{\oplus m\tbinom{r}{2}},\\ 
	H=\bigwedge^{\ell-1}L_{\varphi_2}^2(\Omega)&^{\oplus m},\ D=\left\{v\in\bigwedge^{\ell-1}\Gamma\left(\Omega,\mathcal{S}_P\right)^{\oplus m}\ |\ \iota_{h}v=0\right\},\\
	R=\iota_{h}, \ T \ {\rm and}& \ S\ \text{are given by}\ (3.7)\ \text{in \cite{JYY}}\ \text{with}\ q=0.
\end{align*}
where 
\begin{align}\label{coo}
	\phi_1,\phi_2,\phi_3\ \text{satisfy the hypotheses of Lemma 3.3 in \cite{JYY}},\ \varphi_2-\phi_1=\log|h|^2.
\end{align}
It's easy to see that ${\rm Ker}T\subseteq\bigwedge^\ell\Gamma\left(\Omega,\mathcal{S}_P\right)^{\oplus m}$, and therefore $R\left({\rm Ker}T\right)\subseteq D$. For $v\in\bigwedge^{\ell-1}\Gamma\left(\Omega,\mathcal{S}_P\right)^{\oplus m}$ we have
\begin{align}\label{aoc}
	R^*v:=e^{\phi_1-\varphi_2}\bar{h}\wedge v:=e^{\phi_1-\varphi_2}\bigg(\sum_{j\in J}(-1)^{\left(j,J\setminus j\right)}\bar{h}_jv_{J\setminus j}\bigg)_{|J|=\ell}.
\end{align}

\begin{remark}
	As we only make use of the first two differential operators of the complex (1.4) in \cite{JYY}, $S\circ T=0$ always holds without the assumption (A2) there.
\end{remark}

In the special case where $\eta=0$, we have the following
\begin{lemma}\label{division thm.0}
	Assume that $P$ is an elliptic operator, $\Omega$ is a $P$-convex domain and that $\eta=0$.
	If $h_1,\cdots,h_m\in\Gamma(\Omega,\mathcal{S}_P)$ are functions with no common zeros, then for any $P$-convex function $\phi$ in $\Omega$ and $f\in\Gamma(\Omega,\mathcal{I}_\phi^{\ell-1})$ satisfying $\iota_hf=0$, there exists some $u\in\Gamma(\Omega,\mathcal{I}_{\phi}^\ell)$ such that $\iota_hu=f$ where $1\leq \ell\leq m$.
\end{lemma}

Before giving the proof of theorem, let's recall the following generalization(see \cite{Jqc12}) of Skoda's inequality(see \cite{Sh72}):
\begin{lemma}\label{la.}
	Given $a_i,b_{i\alpha},c_{J,\alpha}\in \mathbb{C}\ (|J|=\ell)$. Then for any strictly increasing multi-index $I$ with length $\ell-1$ we have
	\begin{align*}
		\bigg|\sum_{i\notin I}\bar{a}_j\left(a_jb_{i\alpha}-a_ib_{j\alpha}\right)c_{iI,\alpha}\bigg|^2\leq \lambda\sum_{i=1}^m|a_i|^2\sum_{i\notin I}\sum_{j<k}\bigg|\left(a_jb_{k\alpha}-a_kb_{j\alpha}\right)c_{iI,\alpha}\bigg|^2,
	\end{align*}
	where the constant $\lambda=\left\{\begin{array}{rcl}{\rm min}\{m-1,r\},  &\text{if} \ \ell=1, \\ {\rm min}\{m-\ell+1,r\}, &\text{if} \ \ell\geq2.\end{array}\right.$
\end{lemma}

\begin{proof}[\bf Proof\ of\ Lemma\ \ref{division thm.0}]
	Since $\phi$ is $P$-convex function, by the standard arguments of smooth approximation for $\phi$ and taking weak limit, 
	we can assume without loss of generality that $\phi\in C^2\left(\Omega\right)$. 
	
	Let $v\in D,\ w\in{\rm Ker}S\cap{\rm Dom}\left(T^*\right)$, then	\begin{align}\label{de1}
		\|R^*v+T^*w\|_{\phi_1}^2&=\|R^*v\|_{\phi_1}^2+2{\rm Re}\left(R^*v,T^*w\right)_{\phi_1}+\|T^*w\|_{\phi_1}^2\nonumber\\
		&=:{\rm I}+{\rm II}+{\rm III}.
	\end{align}
	We will handle these terms separately. Firstly, by (\ref{coo}) and (\ref{aoc}) we know that for the first term I,
	\begin{align}\label{de2}
		{\rm I}=&\int_{\Omega}|\bar{h}\wedge v|^2e^{\phi_1-2\varphi_2}\nonumber\\
		=&\sum_{|J|=\ell}\sum_{i,j\in J}\left(-1\right)^{\left(i,J\setminus i\right)+\left(j,J\setminus j\right)}\int_\Omega\left<\bar{h}_iv_{J\setminus i},\bar{h}_jv_{J\setminus j}\right>e^{\phi_1-2\varphi_2}\nonumber\\
		=&\sum_{|L|=\ell-1}\sum_{i\notin L}\int_\Omega|h_i|^2|v_{L}|^2e^{\phi_1-2\varphi_2}-\sum_{|K|=\ell-2}\sum_{i\neq j\notin K}\left<\bar{h}_iv_{jK},\bar{h}_jv_{iK}\right>e^{\phi_1-2\varphi_2}\nonumber\\
		=&\sum_{|L|=\ell-1}\sum_{i=1}^m\int_\Omega|h_i|^2|v_{L}|^2e^{\phi_1-2\varphi_2}-\sum_{|K|=\ell-2}\sum_{i,j\notin K}\int_\Omega\left<\bar{h}_i\bar{v}_{iK},\bar{h}_j\bar{v}_{jK}\right>e^{\phi_1-2\varphi_2}\nonumber\\
		=&\sum_{|L|=\ell-1}\sum_{i=1}^m\int_\Omega|h_i|^2|v_{L}|^2e^{\phi_1-2\varphi_2}\nonumber\\
		=&\|v\|_{\varphi_2}^2.
	\end{align}
	In the third line, we changed the indices of summation by letting $ 
	L=J\setminus i\ \text{and} \ K=J\setminus{\lbrace i,j\rbrace} $. The last but one line is valid since $v\in D$.
	
	Let $w=\left(w_{J,\alpha}\right)_{|J|=\ell, 1\leq \alpha\leq r}\in{\rm Ker}S\cap{\rm Dom}\left(T^*\right)\cap\mathcal{D}(\Omega)^{\oplus\tbinom{r}{1}}$, then we can use (\ref{coo}) to rewrite the second term II as follows.
	\begin{align*}
		{\rm II}&=2{\rm Re}\int_\Omega\left<|h|^{-2}\bar{h}\wedge v,T^*w\right>e^{-\phi_1}\nonumber\\
		&=2{\rm Re}\int_\Omega\left<T\left(|h|^{-2}\bar{h}\wedge v\right),w\right>e^{-\phi_2}\nonumber\\
		&=2{\rm Re}\sum_{|J|=\ell}\sum_{j\in J}\left(-1\right)^{\left(j,J\setminus j\right)}\int_\Omega\left<P_\alpha\left(|h|^{-2}\bar{h}_j\right)v_{J\setminus j},w_{J,\alpha}\right>e^{-\phi_2}\nonumber\\
		&=2{\rm Re}\sum_{|I|=\ell-1}\sum_{j\notin I}\int_\Omega\left<P_\alpha\left(|h|^{-2}\bar{h}_j\right)v_I,w_{jI,\alpha}\right>e^{-\phi_2}.
	\end{align*}
	By using the Cauchy-Schwarz inequality, we have for any constant $C>1$
	\begin{align}\label{de3}
		{\rm II}\geq&-\frac{1}{C}\|v\|_{\varphi_2}^2-C\sum_{|I|=\ell-1}\int_\Omega|h|^2\bigg|\sum_{j\notin I}P_\alpha\left(|h|^{-2}\bar{h}_j\right)\bar{w}_{jI,\alpha}\bigg|^2e^{-\phi_3}\nonumber\\
		=&-\frac{1}{C}\|v\|_{\varphi_2}^2-C\sum_{|I|=\ell-1}\int_\Omega|h|^{-6}\bigg|\sum_{i\notin I}h_j\left(\bar{h}_jP_\alpha \bar{h}_i-\bar{h}_iP_\alpha \bar{h}_j\right)\bar{w}_{iI,\alpha}\bigg|^2e^{-\phi_3}.
	\end{align}
	Applying Lemma \ref{la.} by letting
	$$a_i=\bar{h}_i,\ b_{i\alpha}=P_\alpha\bar{h}_i,\ c_{iI,\alpha}=\bar{w}_{iI,\alpha},$$
	then
	\begin{align}\label{de4}
		&|h|^{-6}\bigg|\sum_{i\notin I}h_j\left(\bar{h}_jP_\alpha\bar{h}_i-\bar{h}_iP_\alpha\bar{h}_j\right)\bar{w}_{iI,\alpha}\bigg|^2\nonumber\\
		\leq&\lambda|h|^{-4}\sum_{i\notin I}\sum_{j<k}\bigg|\left(\bar{h}_jP_\alpha\bar{h}_{k}-\bar{h}_kP_\alpha\bar{h}_{j}\right)\bar{w}_{iI,\alpha}\bigg|^2\nonumber\\
		=&\lambda|h|^{-2}\sum_{i\notin I}\sum_{j=1}^m\bigg|P_\alpha\bar{h}_{j}\bar{w}_{iI,\alpha}\bigg|^2-\lambda|h|^{-4}\sum_{i\notin I}\bigg|h_jP_\alpha\bar{h}_j\bar{w}_{iI,\alpha}\bigg|^2\nonumber\\
		=&\lambda\sum_{i\notin I}Q_{\log|h|^2,\mathfrak{p}}\left(w_{iI},w_{iI}\right),
	\end{align}
	where the penultimate line follows from Lagrange identity, and the last line holds since:
	\begin{align*}
		&P_\alpha\bar{P}_\beta\log|h|^2\bar{w}_{iI,\alpha}w_{iI,\beta}\\
		=&P_\alpha\left(|h|^{-2}\bar{h}_k\bar{P}_\beta h_k\right)\bar{w}_{iI,\alpha}w_{iI,\beta}\\
		=&\left(-|h|^{-4}h_jP_\alpha\bar{h}_j\bar{h}_k\bar{P}_\beta h_k+|h|^{-2}P_\alpha\bar{h}_k\bar{P}_\beta h_k+|h|^{-2}\bar{h}_kP_\alpha\bar{P}_\beta h_k\right)\bar{w}_{iI,\alpha}w_{iI,\beta}\\
		=&-|h|^{-4}|h_jP_\alpha\bar{h}_j\bar{w}_{iI,\alpha}|^2+|h|^{-2}\sum_{j=1}^{m}|P_\alpha\bar{h}_j\bar{w}_{iI,\alpha}|^2+|h|^{-2}\bar{h}_k\left[P_\alpha,\bar{P}_\beta\right]h_k\bar{w}_{iI,\alpha}w_{iI,\beta}\\
		=&-|h|^{-4}|h_jP_\alpha\bar{h}_j\bar{w}_{iI,\alpha}|^2+|h|^{-2}\sum_{j=1}^{m}|P_\alpha\bar{h}_j\bar{w}_{iI,\alpha}|^2-e_{\alpha\beta}^\gamma\bar{P}_\gamma\log|h|^2\bar{w}_{iI,\alpha}w_{iI,\beta}.
	\end{align*}
	Thus, from (\ref{de3}), (\ref{de4}) and Lemma 3.3 in \cite{JYY}, it follows that for $w\in{\rm Ker}S\cap{\rm Dom}\left(T^*\right)$
	\begin{align}\label{de5}
		{\rm II}\geq-\frac{1}{C}\|v\|_{\varphi_2}^2-C\lambda\sum_{|I|=\ell-1}\sum_{i\notin I}\int_\Omega Q_{\log|h|^2,\mathfrak{p}}\left(w_{iI},w_{iI}\right)e^{-\phi_3}.
	\end{align}
	If $w\in{\rm Ker}S\cap{\rm Dom}\left(T^*\right)$, in view of Proposition 3.1 in \cite{JYY}, there exists a smooth function $\kappa'>0$ such that for any constant $C'>1$
	\begin{align}\label{de6}
		{\rm III}\geq\sum_{|J|=\ell}\int_\Omega\left(\frac{1}{C'}Q_{\phi_1,\mathfrak{p}}\left(w_J,w_J\right)-\kappa'|w_J|^2\right)e^{-\phi_3}.
	\end{align}
	Combining (\ref{de1}), (\ref{de2}), (\ref{de5}) and (\ref{de6}) gives
	\begin{align*}
		\|R^*v+T^*w\|_{\phi_1}^2\geq\frac{C-1}{C}\|v\|_{\varphi_2}^2+\sum_{|J|=\ell}\int_\Omega\left(Q_{\frac{\phi_1}{C'}-C\lambda \ell\log|h|^2,\mathfrak{p}}\left(w_{J},w_{J}\right)-\kappa'|w_J|^2\right)e^{-\phi_3}.
	\end{align*}
	
	Since $\Omega$ is $P$-convex, one can find a positive exhaustion function $\varphi$ and positive continuous function $\theta$ in $\Omega$ such that $Q_{\varphi,\mathfrak{p}}\left(w_J,w_J\right)>\theta|w_J|^2$ holds at any point $\mathfrak{p}\in\Omega$. On the other hand, $f\in\Gamma(\Omega,\mathcal{I}_{\phi}^\ell)$ implies that 
	$$fe^{-\frac{1}{2}\phi}\in L_{loc}^2(\Omega).$$
	Then we can choose convex increasing function $\chi\in C^\infty\left(\mathbb{R}\right)$  such that $\chi\left(\varphi\right)>0$ and
	\begin{align}\label{de7}
		\left\{
		\begin{aligned}
			&\frac{1}{C'}\chi'\left(\varphi\right)\theta|w_J|^2-Q_{\left(C\lambda \ell+\frac{1}{C'}\right)\log|h|^2,\mathfrak{p}}\left(w_{J},w_{J}\right)\geq\kappa'|w_J|^2,\\
			&fe^{-\frac{1}{2}\phi}\in L_{\chi\left(\varphi\right)}^2(\Omega).
		\end{aligned}
		\right.
	\end{align}
	Let
	\begin{align}\label{de8}
		\phi_1:=\chi\left(\varphi\right)+\phi-\log|h|^2,\qquad\varphi_2:=\chi\left(\varphi\right)+\phi.
	\end{align}
	By (\ref{de7}), we have
	\begin{align}\label{de9}
		Q_{\frac{\phi_1}{C'}-C\lambda \ell\log|h|^2,\mathfrak{p}}\left(w_{J},w_{J}\right)\geq\kappa'|w_J|^2.
	\end{align}
	Hence, for any $v\in D, w\in{\rm Ker}S\cap{\rm Dom}\left(T^*\right)$  we obtain
	\begin{align*}
		|\left(f,v\right)_H|^2\leq\|f\|_{\varphi_2}^2\|v\|_{\varphi_2}^2\leq\frac{C}{C-1}\|f\|_{\varphi_2}^2\|R^*v+T^*w\|_{\phi_1}^2.
	\end{align*}
	According to Lemma \ref{riesz}, there exists at least one $u\in\bigwedge^\ell\Gamma(\Omega,\mathcal{S}_P)^{\oplus m}$ such that $\iota_hu=f$ and
	\begin{align}\label{de10}
		\int_\Omega|u|^2e^{-\phi_1}\leq\frac{C}{C-1}\int_\Omega|f|^2e^{-\varphi_2}<+\infty,
	\end{align}
	which gives
	\begin{align*}
		\int_\omega|u|^2e^{-\phi}\lesssim\int_\Omega|u|^2e^{-\phi_1}<+\infty,
	\end{align*}
	where $\omega\Subset\Omega$. Then it follows that $u\in\Gamma(\Omega,\mathcal{I}^\ell_{\phi})$.
\end{proof}

\begin{remark}
	One can use (\ref{Pp}) to remove the assumption $\eta = 0$ in Lemma \ref{division thm.0} to get the following: Assume that $P$ is an elliptic operator, $\Omega$ is a $P$-convex domain and $\phi$ is a $P$-convex function in $\Omega$.
	If the functions $h_1,\cdots,h_m\in\Gamma(\Omega,\mathcal{S}_P)$ have no common zeros, then for any $f\in\Gamma(\Omega,\mathcal{I}_\phi^{\ell-1})$ satisfying $\iota_h^\star f=0$, there exists some $u\in\Gamma(\Omega,\mathcal{I}_{\phi}^\ell)$ such that $\iota_h^\star u=f$ where $1\leq \ell\leq m$.
\end{remark}

We now remove from Lemma \ref{division thm.0} the assumption that $h_1,\cdots,h_m\in\Gamma(\Omega,\mathcal{S}_P)$ have no common zeros.

\begin{prop}\label{de.}
	Assume that $P$ is an elliptic operator, $\Omega$ is a $P$-convex domain and $\phi$ is a $P$-convex function in $\Omega$. Let $\rho>1$ be a constant and $\lambda$ be the constant in Lemma \ref{la.}. 
	If $\eta=0$, then for any $f\in\Gamma\left(\Omega,\mathcal{I}_{\phi+\left(\lambda\ell\rho+1\right)\log|h|^2}^{\ell-1}\right)$ with $\iota_hf=0$, there exists some  $u\in\Gamma\left(\Omega,\mathcal{I}_{\phi+\lambda\ell\rho\log|h|^2}^\ell\right)$ such that $\iota_hu=f$ where  $h_1,\cdots,h_m\in\Gamma(\Omega,\mathcal{S}_P)$ and $1\leq \ell\leq m$.
\end{prop}

\begin{proof}
	Let $S_i:=\{h_i=0\}\ (1\leq i\leq m)$, then $\Omega\setminus S_i$ is $P$-convex. In fact,
	\begin{align*}
		P_\alpha\bar{P}_\beta|h_i|^{-2}\bar{\xi}_\alpha\xi_\beta&=P_\alpha\left(-|h_i|^{-4}\bar{h}_i\bar{P}_\beta h_i\right)\bar{\xi}_\alpha\xi_\beta\\
		&=|h_i|^{-4}|P_\alpha\bar{h}_i\bar{\xi}_\alpha|^2-|h_i|^{-4}\bar{h}_i\left[P_\alpha,\bar{P}_\beta\right]h_i\bar{\xi}_\alpha\xi_\beta\\
		&=|h_i|^{-4}|P_\alpha\bar{h}_i\bar{\xi}_\alpha|^2+|h_i|^{-4}\bar{h}_ie_{\alpha\beta}^\gamma\bar{P}_\gamma h_i\bar{\xi}_\alpha\xi_\beta\\
		&=|h_i|^{-4}|P_\alpha\bar{h}_i\bar{\xi}_\alpha|^2-e_{\alpha\beta}^\gamma\bar{P}_\gamma|h_i|^{-2}\bar{\xi}_\alpha\xi_\beta.
	\end{align*}
	Then
	\begin{align*}
		Q_{|h_i|^{-2},\mathfrak{p}}\left(\xi,\xi\right)=|h_i|^{-4}|P_\alpha\bar{h}_i\bar{\xi}_\alpha|^2\geq0.
	\end{align*}
	On the other hand, since $\Omega$ is $P$-convex, there exists exhasution function $\varphi'\in C^2\left(\Omega\right)$ such that its quadratic form is positive at every point in $\Omega$. Thus,
	\begin{align*}
		\varphi:=\varphi'+|h_i|^{-2}
	\end{align*}
	is the exhasution $C^2$-function with positive quadratic form in $\Omega\setminus S_i$. 
	
	Applying Lemma \ref{division thm.0} to $\Omega\setminus S_i$, there exists $u\in\bigwedge^\ell\Gamma\left(\Omega\setminus S_i,\mathcal{S}_P\right)^{\oplus m}$ such that $\iota_hu=f$. However, if we choose the weight functions as in (\ref{de8}), the estimate (\ref{de10}) may be failed. The reason why we cannot choose such $\chi$ since (\ref{de10}) essentially depends on the first estimate in (\ref{de7}) which is false now. 
	But we can choose convex increasing function $\chi\in C^\infty\left(\mathbb{R}\right)$ such that $\chi\left(\varphi\right)>0$ and 
	\begin{align*}
		\frac{1}{C'}\chi'\left(\varphi\right)\theta\geq\kappa',\quad fe^{-\frac{1}{2}\phi-\left(\lambda\ell\rho+1\right)\log|h|}\in L_{\chi\left(\varphi\right)}^2(\Omega\setminus S_i).
	\end{align*}
	As we have explained in the proof of Lemma \ref{division thm.0}, we can assume that $\phi\in C^2\left(\Omega\right)$. Let
	\begin{align*}
		\phi_1:=\chi\left(\varphi\right)+\phi+CC'\lambda \ell\log|h|^2,\quad\varphi_2:=\chi\left(\varphi\right)+\phi+\left(CC'\lambda \ell+1\right)\log|h|^2,
	\end{align*}
	where $CC'=\rho$. Then for such $\phi_1$ the estimate (\ref{de9}) is fulfilled. Hence
	\begin{align}\label{dte1}
		\int_{\Omega\setminus S_i}|u|^2e^{-\phi_1}\lesssim\int_{\Omega\setminus S_i}|f|^2|h|^{-2\left(\lambda \ell\rho+1\right)}e^{-\phi-\chi\left(\varphi\right)}<+\infty.
	\end{align}
	
	Since $S_i$ is a zero measure set, then by (\ref{dte1}) we know that $u$ is local $L^2$-integrable. Since $P$ is elliptic, as a corollary of the complex Frobenius theorem(\cite{N}), around every $\mathfrak{p}\in \Omega$ there is a coordinate chart
	$$(U;x_1,\cdots, x_{r-s}, y_1, \cdots, y_{r-s}, t_1,\cdots, t_s), \ 2r-s=n$$
	on which
	\begin{equation}\label{B}
		\{P_1,\cdots,P_r\} \ {\rm and} \  \{\partial_{\bar{z}_1},\cdots,\partial_{\bar{z}_{r-s}},\partial_{t_1},\cdots\partial_{t_s}\} \ {\rm span \ the \ same \ subbundle},
	\end{equation} 
	where $z_\mu=x_\mu+\sqrt{-1}y_\mu$ for $1\leq \mu\leq r-s$. According to (\ref{B}) we obtain that $u$ can be treated as a holomorphic function with respect to variables $(z_1,\cdots,z_{r-s})$ locally. By using the $L^2$-extension theorem(see \cite{Sh72}) for holomorphic functions, one can find some $\tilde{u}\in\bigwedge^\ell\Gamma\left(\Omega,\mathcal{S}_P\right)^{\oplus m}$ such that $\tilde{u}|_{\Omega\setminus S_i}=u$, then the continuity of functions gives $\iota_h\tilde{u}=f$. Then it follows from (\ref{dte1}) again that $\tilde{u}\in\Gamma\left(\Omega,\mathcal{I}_{\phi+\lambda\ell\rho\log|h|^2}^\ell\right)$.
\end{proof}

Now we are in a position to state our main result for division problem.

\begin{thm}\label{division thm}
	Assume that $P$ is an elliptic operator, $\Omega$ is a $P$-convex domain, and $\phi$ is a $P$-convex function in $\Omega$. Let $h_1,\cdots,h_m\in\Gamma(\Omega,\mathcal{S}_P)$ and $\rho>1$ be a constant, then for any $f\in\Gamma\left(\Omega,\mathcal{I}_{\phi+\left(\lambda\ell\rho+1\right)\log|h|^2}^{\ell-1}\right)$  with $\iota_h^\star f=0$, there exists some $u\in\Gamma\left(\Omega,\mathcal{I}_{\phi+\lambda\ell\rho\log|h|^2}^\ell\right)$ such that $\iota_h^\star u=f$ where  $\lambda$ is the constant in Lemma \ref{la.} and $1\leq \ell\leq m$.
\end{thm}
\begin{proof}
	The equality (\ref{Pp}) gives $e^\eta h\in\Gamma(\Omega,{S}_p)^{\oplus m}$ and $$e^\eta f\in\Gamma\left(\Omega,\mathcal{I}_{\phi+\left(\lambda\ell\rho+1\right)\log|h|^2}^{\ell-1}\right).$$ Since $\Omega$ is $p$-convex, according to Proposition \ref{de.} we know that there exists some $\tilde{u}\in\Gamma\left(\Omega,\mathcal{I}_{p,\phi+\lambda\ell\rho\log|h|^2}^\ell\right)$ such that $\iota_{e^\eta h}\tilde{u}=e^\eta f$, where $\mathcal{I}_{p,\phi+\lambda\ell\rho\log|h|^2}^\ell$ is the subsheaf of $\bigwedge^\ell\mathcal{S}_{p}^{\oplus m}$ defined as Definition \ref{ideal}. Thus, $u:=e^{-\eta}\tilde{u}\in\Gamma\left(\Omega,\mathcal{I}_{\phi+\lambda\ell\rho\log|h|^2}^\ell\right)$  satisfies $\iota_h^\star u=f$.
\end{proof}

\section{A generalization of Nadel's coherence theorem}
In this section, we will show, under certain conditions, the $\left(\mathcal{S}_P,+,\star\right)$-coherence of the subsheaf $\mathcal{I}_\phi^\ell\subseteq\bigwedge^\ell\mathcal{S}_P^{\oplus m}$ for $0\leq\ell\leq m$. The following condition was introduced in \cite{JYY}
\begin{equation}\label{A2*}
	{\rm rank}_{\mathbb{C}}\left(a_\alpha^\sigma(\mathfrak{p})\right)_{1\leq \sigma\leq n\atop1\leq \alpha\leq r}=r\ \text{hold for all} \ \mathfrak{p}\in\Omega.
\end{equation}

For later use we need the following weighted $L^2$-estimate.
\begin{lemma}\label{ewe.}
	Besides (\ref{integrability condition}) and (\ref{condition of the Lie braket}), we also assume that $P$ satisfies (\ref{A2*}). Let $\mathfrak{p}_0\in \Omega$ be an arbitrary point, there exists a $P$-convex neighborhood $U\Subset\Omega$ of $\mathfrak{p}_0$  and a constant $C_U>0$ such that for any $P$-convex function $\phi$ in $U$ and $f\in L_\phi^2(U)^{\oplus\tbinom{r}{1}}$ satisfying $ \mathcal{P}_{2}f=0 $, one can find some $ u\in L^2_{\phi}(U)$ such that $ \mathcal{P}_{1}u=f $ and
	\begin{equation}\label{e1}
		\int_U|u|^2e^{-\phi}\leq C_U\int_U|f|^2e^{-\phi}.
	\end{equation}
\end{lemma}

\begin{proof}
	Without loss of generality, we assume $\mathfrak{p}_0=0$.  By (\ref{local}), we know that all of the eigenvalues of the quadratic form $Q_{|x(\mathfrak{p})|^2,\mathfrak{p}}$ are bounded from below by a positive constant on  each sufficiently small $P$-convex neighborhood $U\Subset\Omega$ of $0$.
	
	We first consider the case $\phi\in C^2(U)$. Since $U$ is $P$-convex, there exists an exhaustion function $\varphi\in C^2\left(U\right)$ such that $Q_{\varphi,\mathfrak{p}}$ is positive definite  at all $\mathfrak{p}\in U$.  Let $\lbrace\eta_\iota\rbrace_{\iota=1}^{\infty} \subseteq \mathcal{D}(U) $ be a sequence of real valued functions such that $ 0\leq\eta_\iota\leq1 $ for all $ \iota $, and $ \eta_\iota=1 $ on any compact subset of $ U $ when $ \iota $ is large.  For an arbitrarily given real number $c$, Lemma 3.3 in \cite{JYY} implies that there is a $\tau\geq0$ satisfying 
	\begin{align}\label{tau1}
		\tau|_{{}_{U_{c}}}=0\ \text{and}\ |\sigma(\mathcal{P}_{q})(\cdot,d\eta_\iota)|^2\leq e^{\tau},\ \iota=1,2,3,\cdots,
	\end{align}
	where $U_c:=\{\mathfrak{p}\in U\ |\ \varphi<c\}\Subset U.$
	
	We will apply Proposition 3.1 in \cite{JYY}  to the $P$-convex open subset $U$ with
	\begin{align*}
		\phi_3=\phi+\chi\left(\varphi\right)+A|x(\mathfrak{p})|^2,\ \phi_2=\phi_3-\tau,\ \phi_1=\phi_3-2\tau,
	\end{align*}
	where $\tau$ satisfies (\ref{tau1}), $\chi\geq0$ is a convex increasing smooth function satisfying $\chi|_{\left(-\infty,c\right)}\equiv0$ and $A\gg 0$ is a constant such that
	\begin{equation*}
		\left\{
		\begin{aligned}
			\phi_3-2\tau\geq\phi,\ AQ_{|x(\mathfrak{p})|^2,\mathfrak{p}}\left(\xi,\xi\right)\geq\left(\kappa^2+1\right)|\xi|^2,\\
			\chi'\left(\varphi\right)Q_{\varphi,\mathfrak{p}}\left(\xi,\xi\right)\geq 4|\sigma\left({}^t\mathcal{P}_q\right)\left(\mathfrak{p},d\tau\right)|^2|\xi|^2.
		\end{aligned}
		\right.
	\end{equation*}
	We can  choose such a constant $A$ because $\kappa^2$ can be bounded uniformly in $U\Subset\Omega$. Thus, it follows that for any $g\in{\rm Dom}\left(T^*\right)\cap{\rm Dom}\left(S\right)$
	\begin{align*}
		\frac{1}{2}\int_U|g|^2e^{-\phi_3}\leq\|T^*g\|_{\phi_1}^2+\|Sg\|_{\phi_3}^2.
	\end{align*}
	The Cauchy-Schwarz inequality implies	\begin{align*}
		|\left(f,g\right)_{\phi_2}|^2\leq\int_U|f|^2e^{-\phi}\int_U|g|^2e^{-\phi_3}\leq2\int_U|f|^2e^{-\phi}\left(\|T^*g\|_{\phi_1}^2+\|Sg\|_{\phi_3}^2\right),
	\end{align*}
	which gives 
	\begin{align}\label{fc.}
		|\left(f,g\right)_{\phi_2}|^2\leq2\int_U|f|^2e^{-\phi}\|T^*g\|_{\phi_1}^2\ {\rm for}\ g\in{\rm Dom}\left(T^*\right)\cap{\rm Ker}\left(S\right).
	\end{align}
	As $f\in{\rm Ker}\left(S\right)$, if $g\in{\rm Dom}\left(T^*\right)\cap{\rm Ker}\left(S\right)^\perp$ then $\left(f,g\right)_{\phi_2}=0$. 
	Thus the estimate (\ref{fc.}) holds for any $g\in{\rm Dom}\left(T^*\right)$. According to Lemma 3.2 in \cite{JYY}, there exists some $u_c\in L_{\phi_1}^2(U)$ such that $\mathcal{P}_1u_c=f$ and
	\begin{align}\label{e.}
		\int_U|u_c|^2e^{-\phi_1}\leq2\int_U|f|^2e^{-\phi}.
	\end{align}
	Choosing a subsequence $\{c_\iota\}_{\iota=1}^\infty$ such that $\{u_{c_\iota}\}_{\iota=1}^\infty$ has a weak limit $u$, we know  by (\ref{e.})  that for any $c$,
	\begin{align*}
		\int_{U_c}|u|^2e^{-\phi-A|x(\mathfrak{p})|^2}=\int_{U_c}|u|^2e^{-\phi_1}\leq\overline{\lim}\int_U|u_{c_\iota}|^2e^{-\phi_1}\leq2\int_U|f|^2e^{-\phi},
	\end{align*}
	which gives 
	\begin{align*}
		\int_{U}|u|^2e^{-\phi}\leq C_U\int_U|f|^2e^{-\phi},
	\end{align*}
	where $C_U$ is the constant only depends on $U$.
	
	Now, we consider the general case where $\phi$ is not necessarily a $C^2$-function. According to the definition of $P$-convexity, on each $U_c\Subset\Omega$, $\phi$ is the limit of  a decreasing sequence  $\{\phi_\iota\}_{\iota=1}^\infty\subseteq C^2\left(U_c\right)$  of $P$-convex functions. Then by the first part in the proof we know that there is $u_{c,\iota}\in L_{\phi_\iota}^2(U_{c})$ such that $\mathcal{P}_1u_{c,\iota}=f$ in $U_{c}$ and
	\begin{align*}
		\int_{U_{c}}|u_{c,\iota}|^2e^{-\phi_\iota}\leq C_{U}\int_{U_c}|f|^2e^{-\phi_\iota}\leq C_{U}\int_U|f|^2e^{-\phi}.
	\end{align*}
	Thus, for any $\upsilon\in\mathbb{Z}_+$, there exists a weak limit $u_{c}$ of the sequence $\{u_{c,\iota}\}_{\iota=\upsilon}^\infty$ in $L^2\left(U_{c}\right)$ which is independent of $\upsilon$ such that $\mathcal{P}_1u_{c}=f$ and
	\begin{align*}
		\int_{U_{c}}|u_{c}|^2e^{-\phi_\upsilon}\leq\overline{\lim}\int_{U_{c}}|u_{c,\iota}|^2e^{-\phi_\upsilon}\leq\sup_{\iota>\upsilon}\int_{U_{c}}|u_{c,\iota}|^2e^{-\phi_\upsilon}\leq C_U\int_U|f|^2e^{-\phi}.
	\end{align*}
	It follows that
	\begin{align*}
		\int_{U_{c}}|u_{c}|^2e^{-\phi}\leq C_U\int_U|f|^2e^{-\phi}.
	\end{align*}
Hence the weak limit $u\in L^2_{loc}\left(\Omega\right)$ of $\{u_c\}_c$ satisfying $\mathcal{P}_1u=f$ and the desired estimate (\ref{e1}).
\end{proof}

\begin{remark}
	The above lemma holds for each $\mathcal{P}_q$ in \cite{JYY} for $1\leq q\leq r$. By using (\ref{local}), we can weaken the condition (\ref{A2*}) in Lemma \ref{ewe.} accordingly as follows:\\
	\begin{center}
		the assumption  (A2)  in \cite{JYY}  is fulfilled, and ${\rm rank}_{\mathbb{C}}\left(a_\alpha^\sigma(\mathfrak{p})\right)_{1\leq \sigma\leq n\atop1\leq \alpha\leq r}\geq r-q+1.$
	\end{center}
\end{remark}

As in the proof of the division theorem, we first restrict to the special case in which $\eta=0$.
\begin{lemma}\label{coherence 0.}
	Assume that $P=(P_1,\cdots,P_r)$ is an elliptic operator satisfying (\ref{A2*}) and $\phi$ is a $P$-convex function in $\Omega$. If $\eta=0$, then $\mathcal{I}_\phi^\ell$ is $\mathcal{S}_P$-coherent for each $0\leq\ell\leq m$.
\end{lemma}

\begin{proof}
	When $\ell=0$ and $m=1$, $\mathcal{I}_{\phi}^0$ is an ideal sheaf of $\mathcal{S}_P$, we will simply denote $\mathcal{I}_{\phi}^0$ by $\mathcal{I}_{\phi}$ in this special case. By definition, we have $\mathcal{I}_{\phi}^\ell=\bigwedge^\ell\mathcal{I}_{\phi}^{\oplus m}$ which in turn implies that the coherence of $\mathcal{I}_{\phi}^\ell$  can be deduced from that of the ideal sheaf $\mathcal{I}_\phi$.	
	
	Now we begin to prove the coherence of ideal sheaf $\mathcal{I}_\phi$. With the estimate (\ref{e1}) at hand, we can proceed by employing Nadel's argument for the coherence of multiplier ideal sheaves of complex structures(see \cite{Nam90}). By the complex Frobenius theorem(\cite{N}), around every $\mathfrak{p}\in \Omega$ there is a  small coordinate chart
	$$(U;z,t):=(U;x_1+\sqrt{-1}y_1,\cdots, x_{r-s}+\sqrt{-1}y_{r-s}, t_1,\cdots, t_s),\ 2r-s=n$$
	on which conditions in Lemma \ref{ewe.} and (\ref{B}) are fulfilled. We claim that 
\begin{equation}\label{convexity log}\log|z|^2\ \text{is}\ P\text{-convex on}\ U,\end{equation}
where $|z|^2=|z_1|^2+\cdots +|z_{r-s}|^2.$ This is essentially a consequence of the assumption (\ref{condition of the Lie braket}). In fact, (\ref{B}) allows us to rewrite the differential operators $P_\alpha$'s as follows
$$P_\alpha\equiv b_\alpha^\mu\partial_{\bar{z}_\mu}\ {\rm mod}\big\{\partial_{t_1},\cdots,\partial_{t_s}\big\},$$
where $b_\alpha^\mu\in C^\infty\left(U\right), 1\leq\mu\leq r-s, 1\leq\alpha\leq r.$ Substituting the above expressions of $P_\alpha$'s into (\ref{condition of the Lie braket}) and comparing coefficients of $\partial_{\bar{z}_\nu}$, we obtain
\begin{equation}\label{pbeb}
	P_\alpha\overline{b_\beta^\nu}+e_{\alpha\beta}^\gamma\overline{b_\gamma^\nu}=0
\end{equation}
for $1\leq \alpha,\beta\leq r, 1\leq \nu\leq r-s$. Let $\epsilon >0$ be a constant, we know by (\ref{pbeb}) 
\begin{align*}
	&\left(P_\alpha\bar{P}_\beta+e_{\alpha\beta}^\gamma\bar{P}_\gamma\right)\log\left(|z|^2+\epsilon\right)\\ =&P_\alpha\bigg(\frac{1}{|z|^2+\epsilon}\overline{b_\beta^\nu z_\nu}\bigg)+e_{\alpha\beta}^\gamma\frac{\overline{b_\gamma^\nu z_\nu}}{|z|^2+\epsilon}\\ =&b_\alpha^\mu\overline{b_\beta^\nu}\frac{\delta_{\mu\nu}\left(|z|^2+\epsilon\right)-z_\mu\bar{z}_\nu}{\left(|z|^2+\epsilon\right)^2}+\frac{1}{|z|^2+\epsilon}\left(P_\alpha\overline{b_\beta^\nu}+e_{\alpha\beta}^\gamma\overline{b_\gamma^\nu}\right)\bar{z}_\nu\\
	=&b_\alpha^\mu\overline{b_\beta^\nu}\frac{\delta_{\mu\nu}\left(|z|^2+\epsilon\right)-z_\mu\bar{z}_\nu}{\left(|z|^2+\epsilon\right)^2},
\end{align*}
and therefore $\log|z|^2$, as the decreasing limit of $\log\left(|z|^2+\epsilon\right)$, is a $P$-convex function on $U$.

As another consequence of (\ref{B}), each $f_\mathfrak{p}\in\mathcal{S}_{P,\mathfrak{p}}$ is a germ of holomorphic functions in variables $z_\mu$$(1\leq\mu\leq r-s)$, and  $\mathcal{S}_P$ is thus a coherent sheaf of rings. Set $L_{\phi}^2\left(U,\mathcal{S}_P\right)=L_{\phi}^2(U)\cap\Gamma\left(U,\mathcal{S}_P\right)$, it follows from the Noetherian property of coherent sheaves(\cite{GR}, page 111) that $L_{\phi}^2\left(U,\mathcal{S}_P\right)$ generates a coherent ideal sheaf $\mathcal{I}\subseteq\mathcal{S}_P$. Obviously, we have $\mathcal{I}\subseteq\mathcal{I}_\phi$. In order to prove the reverse, according to Krull's Intersection Lemma(c.f. \cite{AM69}), it suffices to prove for every $\iota\in\mathbb{Z}_{\geq 0}, \ \mathcal{I}_{\mathfrak{p}}+\mathfrak{m}_{P,\mathfrak{p}}^{\iota}\mathcal{I}_{P,\mathfrak{p}}\supseteq\mathcal{I}_{P,\mathfrak{p}}$ where $\mathfrak{p}\in U$ and $\mathfrak{m}_{P,\mathfrak{p}}$ is the maximal ideal of $\mathcal{S}_{P,\mathfrak{p}}$.  In view of the Artin-Rees Lemma(see \cite{AM69}), there exists an integer $\iota_0\gg 0$ such that 
$\mathfrak{m}_{P,\mathfrak{p}}^\iota\cap\mathcal{I}_{P,\mathfrak{p}}=\mathfrak{m}_{P,\mathfrak{p}}^{\iota-\iota_0}\left(\mathfrak{m}_{P,\mathfrak{p}}^{\iota_0}\cap\mathcal{I}_{P,\mathfrak{p}}\right)\subseteq\mathfrak{m}_{P,\mathfrak{p}}^{\iota-\iota_0}\mathcal{I}_{P,\mathfrak{p}}$ for all $\iota\geq\iota_0$, it remains to show 
\begin{equation}\label{desired}
	\mathcal{I}_{\mathfrak{p}}+\mathfrak{m}_{P,\mathfrak{p}}^{\iota+\iota_0}\cap\mathcal{I}_{P,\mathfrak{p}}\supseteq\mathcal{I}_{P,\mathfrak{p}}\ {\rm for}\ \iota\in\mathbb{Z}_{\geq 0}.
\end{equation}
	
	For each $f_\mathfrak{p}\in\mathcal{I}_{P,\mathfrak{p}}$, there is, by definition, a neighborhood $V\subseteq U$ of $\mathfrak{p}$ such that $f\in L^2_\phi(V,\mathcal{S}_P)$. For simplicity, we identify $U$ with an open subset of $\mathbb{C}^{r-s}\times\mathbb{R}^s$ and $\mathfrak{p}=(\mathfrak{p}',\mathfrak{p}'')\in V=V'\times V''\subseteq\mathbb{C}^{r-s}\times\mathbb{R}^s$ is the origin. Let $\chi\in \mathcal{D}(V')$ be a compactly supported function which is identically equal to 1 in a small neighborhood of $\mathfrak{p}'$.  As previously mentioned, $f$ is independent of $(t_1,\cdots,t_s)\in V''$. So, it's clear that $fP\chi$ has a natural smooth extension to $U$ and for any $\iota\in\mathbb{Z}_+$
	\begin{align*}
		\int_U |fP\chi&|^2|z|^{-2\left(\iota+\iota_0+r-s-1\right)}e^{-\phi}  \lesssim\int_{V}|f|^2e^{-\phi}<+\infty, \\
		&\mathcal{P}_2\left(fP\chi\right)=\mathcal{P}_2\circ\mathcal{P}_1(f\chi)=0.
	\end{align*}
	Note that $r-s\geq 0$, (\ref{convexity log}) implies that the following function is $P$-convex function on $U$ for any $\iota\in\mathbb{Z}_{\geq 0}$ $$\phi+\left(\iota+\iota_0+r-s-1\right)\log|z|^2.$$
	Consequently, Lemma \ref{ewe.} gives some $ u\in L^2_{\phi}(U)$ solving $Pu= fP\chi$ with
	\begin{align}\label{eq.}
		\int_U |u|^2|z|^{-2\left(\iota+\iota_0+r-s-1\right)}e^{-\phi}&\lesssim\int_U |fP\chi|^2|z|^{-2\left(\iota+\iota_0+r-s-1\right)}e^{-\phi}<+\infty.
	\end{align}
	Set $\tilde{f}:=\chi f-u$, then $\tilde{f}\in L_{\phi}^2\left(U,\mathcal{S}_P\right)$  which implies $\tilde{f}_\mathfrak{p}\in\mathcal{I}_{\mathfrak{p}}$. As $\chi$ is identically equal to 1 in a neighborhood of $\mathfrak{p}'$,  we also have $u_\mathfrak{p}\in\mathcal{I}_{P,\mathfrak{p}}$. Moreover, (\ref{eq.}) implies $u_\mathfrak{p}\in\mathfrak{m}_{P,\mathfrak{p}}^{\iota+\iota_0}$. In summary, we have proved for any $\iota\in\mathbb{Z}_{\geq 0}$
	$$f_\mathfrak{p}=\tilde{f}_{\mathfrak{p}}+u_\mathfrak{p}\in  \mathcal{I}_{\mathfrak{p}}+\mathfrak{m}_{P,\mathfrak{p}}^{\iota+\iota_0}\cap\mathcal{I}_{P,\mathfrak{p}},$$
	and therefore the desired relation (\ref{desired}) follows.
\end{proof}

When $\Omega$ is a $P$-convex domain, we can remove the assumption $\eta=0$ from Lemma \ref{coherence 0.} and reach the coherence of $\mathcal{I}_\phi^\ell$ in general which extends Nadel's coherence theorem from complex structures to overdetermined systems of partial differential equations.
\begin{prop}\label{coherent}
	Let $P=(P_1,\cdots,P_r)$ be an elliptic operator satisfying (\ref{A2*}), $\Omega$ be a $P$-convex domain. Then for every $P$-convex function $\phi$ on $\Omega$, the $\left(\mathcal{S}_P,+,\star\right)$-subsheaf $\mathcal{I}_\phi^\ell$ of $\bigwedge^\ell\mathcal{S}_P^{\oplus m}$ is coherent for $0\leq\ell\leq m$.
\end{prop}

\begin{proof}
	As we have explained in the proof of Lemma \ref{coherence 0.}, it suffices to prove the result for $\ell=0$ and $m=1$. Lemma \ref{coherence 0.} means that for any $\mathfrak{p}\in\Omega$, there exists local neighborhood $U$ of $\mathfrak{p}$ such that the following $\mathcal{S}_{p}$-sequence exact
	\begin{align}\label{c1}
		\mathcal{S}_p|_U^{\oplus a}\longrightarrow\mathcal{S}_p|_U^{\oplus b}\longrightarrow\mathcal{I}_{p,\phi}|_U\longrightarrow0,\quad 1\leq a,b<+\infty,
	\end{align}
	where $\mathcal{I}_{p,\phi}$ is the ideal sheaf of $\mathcal{S}_{p}$ as in the proof of Lemma \ref{coherence 0.}. By the ellipticity assumption of $P$ we know that (\ref{Pp}) holds for some $\eta\in C^\infty\left(\Omega\right)$ and therefore
	\begin{align}\label{I}
		\mathcal{I}_{p,\phi}=\{e^\eta g\ |\ g\in\mathcal{I}_{\phi-2{\rm Re}\eta}\}=\{e^\eta g\ |\ g\in\mathcal{I}_{\phi}\}.
	\end{align}
	It follows from (\ref{Pp}) and (\ref{I}) that the exact sequence (\ref{c1}) is equivalent to the following exact $\left(\mathcal{S}_P,+,\star\right)$-sequence
	\begin{align*}
		\mathcal{S}_P|_U^{\oplus a}\longrightarrow\mathcal{S}_P|_U^{\oplus b}\longrightarrow\mathcal{I}_\phi|_U\longrightarrow0,\quad 1\leq a,b<+\infty.
	\end{align*}
	The proof is thus complete.
\end{proof}

\begin{remark}
	(1) When the system of differential operators $\{P_1,\cdots,P_r\}$ defines a complex structure, then Proposition \ref{coherent} recovers Nadel's coherence theorem(see \cite{Nam90}).
	
	(2) Although the coherence of $\mathcal{S}_P$ follows from that of complex structures, we can not simply deduce Proposition \ref{coherent} from Nadel's coherence theorem for complex structures. The reasons are twofold as follows.
	
	First, we observe that the $P$-convexity does not imply, in general,  the plurisubharmonicity. In fact, let $\phi\in C^2(\Omega)$ be a $P$-convex function, and $(U;x_1,\cdots,x_{r-s},y_1,\cdots,y_{r-s},t_1,\cdots,t_s)$ be a coordinate chart satisfying (\ref{B}), then we have the following on $U$ 
	\begin{align*}
		\partial_{\bar{z}_\mu}=b_\mu^\alpha p_\alpha,
	\end{align*}
	where $p_\alpha$ is the principal part of $P_\alpha$ and $b_\mu^\alpha\in C^\infty\left(\Omega\right),1\leq\mu\leq r-s,1\leq\alpha\leq r$, then for all $1\leq \mu,\nu\leq r-s$
	\begin{align*}
		0=\left[\partial_{\bar{z}_\mu},\partial_{z_\nu}\right]=b_\mu^\alpha\left(p_\alpha\overline{b_\mu^\beta}\right)\bar{p}_\beta-\overline{b_\nu^\beta}\left(\bar{p}_\beta b_\mu^\alpha\right)p_\alpha+b_\mu^\alpha\overline{b_\nu^\beta}\left(d_{\alpha\beta}^\gamma p_\gamma-e_{\alpha\beta}^\gamma\bar{p}_\gamma\right).
	\end{align*}
	Thus,
	\begin{align*}
		\partial_{\bar{z}_\mu}\partial_{z_\nu}\phi&=b_\mu^\alpha p_\alpha\left(\overline{b_\nu^\beta}\bar{p}_\beta\phi\right)\\
		&=b_\mu^\alpha\left(p_\alpha\overline{b_\nu^\beta}\right)\bar{p}_\beta\phi+b_\mu^\alpha\overline{b_\nu^\beta}p_\alpha\bar{p}_\beta\phi\\
		&=b_\mu^\alpha\overline{b_\nu^\beta}\left(p_\alpha\bar{p}_\beta\phi+e_{\alpha\beta}^\gamma\bar{p}_\gamma\phi\right)+\overline{b_\nu^\beta}\left(\bar{p}_\beta b_\mu^\alpha\right)p_\alpha\phi-b_\mu^\alpha\overline{b_\nu^\beta}d_{\alpha\beta}^\gamma p_\gamma\phi.
	\end{align*}
	The difference of two quadratic forms with respect to $P$ and $(\partial_{\bar{z}_1},\cdots\partial_{\bar{z}_{r-s}})$ is given by $\overline{b_\nu^\beta}\left(\bar{p}_\beta b_\mu^\alpha\right)p_\alpha\phi-b_\mu^\alpha\overline{b_\nu^\beta}d_{\alpha\beta}^\gamma p_\gamma\phi$ which causes the ambiguity.
	
	In spite of the difference of convexities defined by $P$ and $(\partial_{\bar{z}_1},\cdots\partial_{\bar{z}_{r-s}})$ respectively, there is another obstruction for deducing Lemma \ref{coherence 0.} from Nadel's coherence theorem. Indeed, for any $\mathfrak{p}\in\Omega$ and $f\in\mathcal{I}_{P,\mathfrak{p}}$ there is  coordinate neighborhood of $\mathfrak{p}$ as above such that $f\in L_\phi^2\left(U,\mathcal{S}_P\right)$, we identify $U$ with an open subset $U'\times U''\subseteq\mathbb{C}^{r-s}\times\mathbb{R}^s$ and write $\mathfrak{p}=(\mathfrak{p}',\mathfrak{p}'')\in U'\times U''$. Even if $\phi(\cdot,t)$ is a plurisubharmonic function in $U'$ for an arbitrarily fixed $t=(t_1,\cdots,t_{s})\in U''$, Nadel's coherence theorem yields local generators $s_{1,t},\cdots,s_{b,t}\in\Gamma\left(V,\mathcal{O}_{\mathbb{C}^{r-s}}\right)$ depending on $t$ such that
	\begin{align*}
		\mathcal{I}_{\phi\left(\cdot,t\right),z}=\mathcal{O}_{V,z}s_{1,t}(z)+\cdots+\mathcal{O}_{V,z}s_{b,t}(z),\quad z\in V,
	\end{align*}
	where $V\subseteq U'$ is a neighborhood of $\mathfrak{p}'$. While the $\mathcal{S}_P$-coherence of $\mathcal{I}_\phi$ requires that all the generators are independent of $t$.
\end{remark}

\end{document}